\documentclass{amsart}

\usepackage{amsmath,latexsym,amssymb,amsthm,graphicx}
\usepackage{mathtools}
\usepackage{xfrac}
\usepackage{stmaryrd}
\usepackage{multirow}
\usepackage{tabularx}
\usepackage{enumitem}
\usepackage{caption}
\setenumerate[1]{label=(\thesection.\arabic*)}

\numberwithin{equation}{section}

\newtheorem{theorem}{Theorem}[section]
\newtheorem{lemma}[theorem]{Lemma}

\newtheorem{corollary}[theorem]{Corollary}

\theoremstyle{definition}

\def\Z{\ensuremath{\mathbb{Z}}}
\def\Q{\ensuremath{\mathbb{Q}}}
\def\R{\ensuremath{\mathbb{R}}}
\def\C{\ensuremath{\mathbb{C}}}
\def\H{\ensuremath{\mathbb{H}}}
\def\A{\ensuremath{\mathbb{A}}}

\newcommand{\pa}[1]{\left(#1\right)}
\newcommand{\cpa}[1]{\left\{#1\right\}}

\newcommand{\tn}[1]{\textnormal{#1}}
\newcommand{\br}[1]{\left[#1\right]}

\newcommand{\tri}[3]{\textnormal{#1--#2--#3}}

\begin{document}
\title{Rational Angled Hyperbolic Polygons}

\author[J.~Calcut]{Jack S. Calcut}
\address{Department of Mathematics\\
         Oberlin College\\
         Oberlin, OH 44074}
\email{jcalcut@oberlin.edu}
%\urladdr{\href{http://www.oberlin.edu/faculty/jcalcut/}{\curl{http://www.oberlin.edu/faculty/jcalcut/}}}

\date{\today}

\begin{abstract}
We prove that every rational angled hyperbolic triangle has transcendental side lengths and that
every rational angled hyperbolic quadrilateral has at least one transcendental side length.
Thus, there does not exist a rational angled hyperbolic triangle or quadrilateral with algebraic side lengths.
We conjecture that there does not exist a rational angled hyperbolic polygon with algebraic side lengths.
\end{abstract}

\maketitle

\section{Introduction}\label{introduction}

Herein, an angle is \emph{rational} provided its radian measure $\theta$ is a rational multiple of $\pi$, written $\theta\in\Q\pi$.
Rational angles may also be characterized as angles having rational degree measure and as angles commensurable with a straight angle.
\emph{Algebraic} means algebraic over $\Q$.
Let $\A\subset\C$ denote the subfield of algebraic numbers.
If $\theta\in\Q\pi$, then $\cos\theta\in\A$ and $\sin\theta\in\A$.\\

In Euclidean geometry, there exists a rich collection of rational angled polygons with algebraic side lengths.
Besides equilateral triangles, one has the triangles in the Ailles rectangle~\cite{ailles} and in the golden triangle (see Figure~\ref{rectangle}).
\begin{figure}[h!]
	\centerline{\includegraphics[scale=1.0]{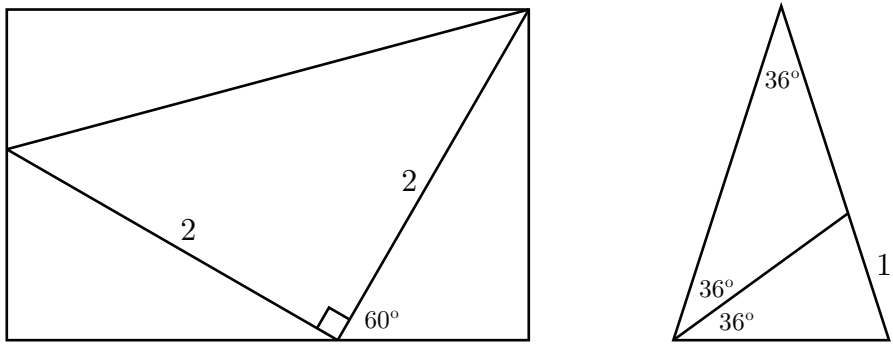}}
	\caption{Ailles rectangle and the golden triangle.}
	\label{rectangle}
\end{figure}
Further, let $\sigma(d)$ denote the number of similarity types of rational angled Euclidean triangles containing a representative with side lengths each of degree at most $d$ over $\Q$.
Then:
\begin{enumerate}
\item $\sigma(d)<\infty$ for each $d\in\Z^+$. A proof of this fact uses three ingredients:
(i) the Euclidean law of cosines,
(ii) if $\varphi$ is Euler's totient function, $n>2$, and $k$ is coprime to $n$, then the degree over $\Q$ of $\cos(2k\pi/n)$ equals $\varphi(n)/2$, and
(iii) $\varphi(n)\geq\sqrt{n/2}$ for each $n\in\Z^+$ (see~\cite{gst}).
\item $\sigma(1)=1$. This class is represented by an equilateral triangle (see~\cite[pp.~228--229]{CG06} or~\cite[p.~674]{gst}).
While Pythagorean triple triangles have integer side lengths, they have irrational acute angles;
a simple approach to this result uses unique factorization of Gaussian integers~\cite{gaussian}.
In fact, the measures of the acute angles in a Pythagorean triple triangle are transcendental in radians and degrees~\cite{gst}.
\item $\sigma(2)=14$. Parnami, Agrawal, and Rajwade~\cite{par} proved this result using Galois theory.
The $14$ similarity types are listed in Table~\ref{gsts.table}.
It is an exercise to construct representatives for these $14$ classes using the triangles above in Figure~\ref{rectangle}.
\begin{table}[h!]\renewcommand{\arraystretch}{1.4}\captionsetup{belowskip=10pt}
\begin{center}
	\begin{tabular}{|c|c|c|c|}
	\hline
		\tri{60}{60}{60}	&	\tri{45}{45}{90}	&	\tri{30}{60}{90}	&	\tri{15}{75}{90}	\\	\hline
		\tri{30}{30}{120}	&	\tri{30}{75}{75}	&	\tri{15}{15}{150}	&	\tri{30}{45}{105}	\\	\hline
		\tri{45}{60}{75}	&	\tri{15}{45}{120}	&	\tri{15}{60}{105}	&	\tri{15}{30}{135}	\\	\hline
		\tri{36}{36}{108}	&	\tri{36}{72}{72} 																					\\	\cline{1-2}
	\end{tabular}
	\caption{The $14$ similarity types (in degrees) of rational angled Euclidean triangles with side lengths each of degree at most two over $\Q$.}
	\label{gsts.table}
\end{center}
\end{table}
\end{enumerate}
A \emph{polar rational polygon} is a convex rational angled Euclidean polygon with integral length sides.
The study of such polygons was initiated by Schoenberg~\cite{schoenberg} and Mann~\cite{mann} and continues with the recent work of Lam and Leung~\cite{ll}.\\

In sharp contrast, there seem to be no well-known rational angled hyperbolic polygons with algebraic side lengths.
Throughout, $\H^2$ denotes the hyperbolic plane of constant Gaussian curvature $K=-1$.
We adopt Poincar\'{e}'s conformal disk model of $\H^2$ for our figures.\\

Below, we prove that every rational angled hyperbolic triangle has transcendental side lengths and every
rational angled hyperbolic quadrilateral (simple, but not necessarily convex) has at least one transcendental side length.
Our basic tactic is to obtain relations using hyperbolic trigonometry to which we may apply the following theorem of Lindemann.

\begin{theorem}[Lindemann~{\cite[p.~117]{niven_1956}}]\label{lindemann}
If $\alpha_1,\alpha_2,\ldots,\alpha_n\in\A$ are pairwise distinct, then $e^{\alpha_1},e^{\alpha_2},\ldots,e^{\alpha_n}$ are linearly independent over $\A$.
\end{theorem}

We conjecture that there does not exist a rational angled hyperbolic polygon with algebraic side lengths.
We ask whether a rational angled hyperbolic polygon can have any algebraic side lengths whatsoever.\\

Throughout, polygons are assumed to be simple, nondegenerate (meaning consecutive triples of vertices are noncollinear), and not necessarily convex.
In particular, each internal angle of a polygon has radian measure in $(0,\pi)\cup(\pi,2\pi)$.
Our closing section discusses generalized hyperbolic triangles with one ideal vertex, spherical triangles, and the dependency of our results on Gaussian curvature.

\section{Hyperbolic Triangles}

Consider a triangle $ ABC \subset \H^2$ as in Figure~\ref{hyp_triangle}.
\begin{figure}[htbp!]
    \centerline{\includegraphics[scale=1.0]{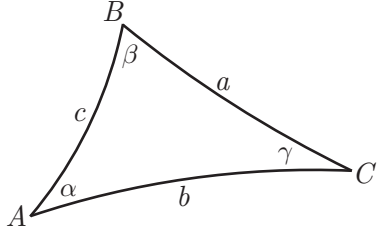}}
    \caption{Hyperbolic triangle $ABC$.}
\label{hyp_triangle}
\end{figure}
Recall the following hyperbolic trigonometric identities (see Ratcliffe~\cite[p.~82]{ratcliffe_2006}) known respectively as
the hyperbolic law of sines, the first hyperbolic law of cosines, and the second hyperbolic law of cosines:
\begin{equation}\label{hlos}
\frac{\sinh a}{\sin\alpha}=\frac{\sinh b}{\sin\beta}=\frac{\sinh c}{\sin\gamma}, \tag{HLOS}
\end{equation}
\begin{equation}\label{hloc1}
\cos\gamma=\frac{\cosh a \cosh b -\cosh c}{\sinh a \sinh b}, \tn{ and} \tag{HLOC1}
\end{equation}
\begin{equation}\label{hloc2}
\cosh c =\frac{\cos \alpha \cos \beta +\cos\gamma}{\sin\alpha \sin\beta}. \tag{HLOC2}
\end{equation}

\begin{lemma}\label{3anglesrational}
Let $ABC \subset \H^2$ be a rational angled triangle labelled as in Figure~\ref{hyp_triangle}. Then, every side length of $ABC$ is transcendental.
\end{lemma}

\begin{proof}
Define $x:=\cosh c > 0$. Then:
\begin{equation}\label{xcoshc}
	x e^0 = \frac{1}{2}e^c + \frac{1}{2}e^{-c}.
\end{equation}
As $\alpha,\beta,\gamma\in\Q\pi$ and $\A$ is a field, \ref{hloc2} implies that $x\in\A$.
As $c>0$, Lindemann's theorem applied to \eqref{xcoshc} implies that $c\not\in\A$.
The proofs for $a$ and $b$ are similar.
\end{proof}

\begin{corollary}
Each regular rational angled hyperbolic polygon has transcendental side length, radius, and apothem.
\end{corollary}

\begin{proof}
Consider the case of a quadrilateral $ABCD\subset\H^2$.
By an isometry of $\H^2$, we may assume $ABCD$ appears as in Figure~\ref{reg_hyp_quad} where $E$ is the origin and $F$ is the midpoint of segment $AB$.
 \begin{figure}[htbp!]
    \centerline{\includegraphics[scale=1.0]{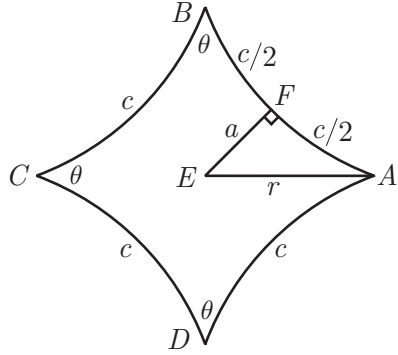}}
    \caption{Regular hyperbolic quadrilateral $ABCD$.}
\label{reg_hyp_quad}
\end{figure}
Triangle $AEF$ has radian angle measures $\theta/2$, $\pi/4$, and $\pi/2$.
Lemma~\ref{3anglesrational} implies that $c/2$, $r$, and $a$ are transcendental.
The cases $n\ne4$ are proved similarly.
\end{proof}

We will use the following lemma in the next section on quadrilaterals.

\begin{lemma}\label{alg_legs}
There does not exist a triangle $ ABC \subset \H^2$, labelled as in Figure~\ref{hyp_triangle}, such that $\gamma\in\Q\pi$, $\alpha+\beta\in\Q\pi$, and $a,b\in\A$.
\end{lemma}

\begin{proof}
Suppose, by way of contradiction, that such a triangle exists.
Define:
\begin{align}\label{xyzZ}
	x:=&\cos\pa{\alpha+\beta}\in(-1,1)\cap\A, \nonumber \\
	y:=&\sin\gamma\in(0,1]\cap\A, \\
	z:=&\cos\gamma\in(-1,1)\cap\A, \tn{ and} \nonumber \\
	Z:=&\cosh c >0. \nonumber
\end{align}
Then, \ref{hloc1} and \ref{hlos} yield:
\begin{multline*}
x=\cos\pa{\alpha+\beta} = \cos\alpha \cos\beta-\sin\alpha \sin\beta \\
 =\frac{Z\cosh b -\cosh a}{\sinh b \sinh c} \cdot \frac{Z\cosh a -\cosh b}{\sinh a \sinh c}- y^2 \frac{\sinh a \sinh b}{\sinh^2 c}.
\end{multline*}
As $\sinh^2 c=Z^2-1$, we obtain:
\begin{multline}\label{den_cleared}
x\pa{Z^2-1} \sinh a \sinh b  \\
= \pa{Z\cosh b -\cosh a}\pa{Z\cosh a -\cosh b}-y^2\sinh^2 a \sinh^2 b.
\end{multline}
By \ref{hloc1}:
\begin{equation}\label{Z_sub}
Z=\cosh a \cosh b -z \sinh a \sinh b.
\end{equation}
Make the substitution~\eqref{Z_sub} in \eqref{den_cleared} and then, for $t=a$ and $t=b$, expand hyperbolics into exponentials using the standard identities:
\[
\cosh t = \frac{1}{2}e^t + \frac{1}{2}e^{-t} \quad \tn{and} \quad \sinh t =\frac{1}{2}e^t - \frac{1}{2}e^{-t}.
\]
Expand out the resulting equation by multiplication, multiply through by $64$, and then subtract the right hand terms to the left side.
Now, collect together terms whose exponentials have identical exponents as elements of $\Z\br{a,b}$.
One obtains:
\begin{equation}\label{rel_R}
	\pa{x-1}\pa{z-1}^2 e^{3a+3b}+\pa{\tn{sum of lower order terms}}=0.
\end{equation}
Equation \eqref{rel_R} is our relation.
Its left hand side, denoted by $R$, is the sum of $25$ terms each of the form $p e^{ma+nb}$ for some $p\in\Z\br{x,y,z}$ and some integers $m$ and $n$ in $\br{-3,3}$.
By definition, \emph{lower order} terms of $R$ have $m<3$ or $n<3$ (or both).
As $a>0$ and $b>0$, Lindemann's theorem applied to \eqref{rel_R} implies that $x=1$ or $z=1$.
This contradicts \eqref{xyzZ} and completes the proof of Lemma~\ref{alg_legs}.
\end{proof}

\section{Hyperbolic Quadrilaterals}

Consider a quadrilateral $ABCD\subset\H^2$.
Whether or not $ABCD$ is convex, at least one diagonal $AC$ or $BD$ of $ABCD$ lies inside $ABCD$
(this holds in neutral geometry).
Relabelling the vertices of $ABCD$ if necessary, we assume $BD$ lies inside $ABCD$ and that $ABCD$ is labelled as in Figure~\ref{hyp_quad}.
\begin{figure}[htbp!]
    \centerline{\includegraphics[scale=1.0]{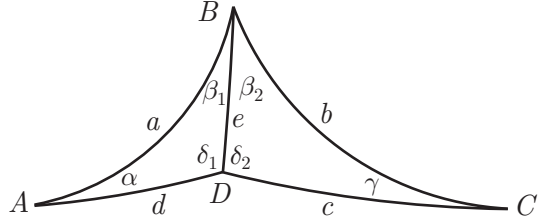}}
    \caption{Quadrilateral $ABCD\subset\H^2$ with internal diagonal $BD$.}
\label{hyp_quad}
\end{figure}
The radian measures of the internal angles of $ABCD$ at $B$ and $D$ are $\beta=\beta_1+\beta_2$ and $\delta=\delta_1+\delta_2$ respectively.

\begin{theorem}\label{quad_thm}
If quadrilateral $ABCD$ is rational angled, then at least one of its side lengths $a$, $b$, $c$, or $d$ is transcendental.
\end{theorem}

\begin{proof}
Suppose, by way of contradiction, that $\alpha,\beta,\gamma,\delta\in\Q\pi$ and $a,b,c,d\in\A$.
Define:
\begin{align}\label{wxyzE}
	w:=&\cos\beta \in (-1,1) \cap\A, \nonumber \\
	x:=&\cos\delta \in (-1,1) \cap\A, \nonumber\\
	y:=&\sin\alpha \in [-1,1]\cap\A -\cpa{0},\\
	z:=&\sin\gamma \in [-1,1]\cap\A -\cpa{0}, \tn{ and} \nonumber\\
	E:=&\cosh e >0. \nonumber
\end{align}
Then, \ref{hloc1} and \ref{hlos} yield:
\begin{multline*}
w=\cos\beta=\cos\pa{\beta_1+\beta_2}=\cos\beta_1\cos\beta_2-\sin\beta_1\sin\beta_2\\
=\frac{E\cosh a -\cosh d}{\sinh a \sinh e} \cdot \frac{E\cosh b -\cosh c}{\sinh b \sinh e} - \frac{\sinh d}{\sinh e} \sin\alpha \frac{\sinh c}{\sinh e} \sin\gamma.
\end{multline*}
As $\sinh^2 e = E^2 -1$, we obtain:
\begin{multline}\label{raw_w}
	w\pa{E^2-1}\sinh a \sinh b \\
	=\pa{E\cosh a -\cosh d}\pa{E\cosh b -\cosh c}-yz\sinh a \sinh b \sinh c \sinh d.
\end{multline}
The analogous calculation beginning with $x=\cos\delta=\cos\pa{\delta_1+\delta_2}$ yields:
\begin{multline}\label{raw_x}
	x\pa{E^2-1}\sinh c \sinh d \\
	=\pa{E\cosh d -\cosh a}\pa{E\cosh c -\cosh b}-yz\sinh a \sinh b \sinh c \sinh d.
\end{multline}
In each of the equations \eqref{raw_w} and \eqref{raw_x}, move all terms to the right hand side and regroup to obtain two quadratics in $E$:
\begin{align*}
	0&=A_1 E^2 + B_1 E +C_1 \tn{ and} \\
	0&=A_2 E^2 + B_2 E +C_2.
\end{align*}
Inspection shows that $B_1=B_2$. Define $B:=B_1=B_2$.\\

We claim that $A_1\neq0$ and $A_2\neq0$.
Suppose, by way of contradiction, that $A_1=0$.
Expanding the hyperbolics in $A_1$ into exponentials and collecting terms by the exponents of their exponentials yields:
\[
	-\frac{1}{4}\pa{w-1}e^{a+b} +\pa{\tn{sum of lower order terms}}=0.
\]
Lindemann's theorem implies $w=1$ which contradicts \eqref{wxyzE}. So, $A_1\neq0$ and similarly $A_2\neq0$.\\

The quadratic formula yields:
\[
0<E=\frac{-B\pm\sqrt{B^2-4 A_1 C_1}}{2A_1} = \frac{-B\pm\sqrt{B^2-4 A_2 C_2}}{2A_2}
\]
where the $\pm$ signs are ambiguous and not necessarily equal.
We have:
\[
-A_2 B \pm A_2 \sqrt{B^2-4 A_1 C_1} = -A_1 B \pm A_1 \sqrt{B^2-4 A_2 C_2}
\]
and so:
\[
A_1 B -A_2 B \pm A_2 \sqrt{B^2-4 A_1 C_1} = \pm A_1 \sqrt{B^2-4 A_2 C_2}.
\]
Squaring both sides yields:
\begin{multline*}
\pa{A_1 B -A_2 B}^2 + A_2^2\pa{B^2-4 A_1 C_1} \pm 2\pa{A_1 B -A_2 B}A_2\sqrt{B^2-4 A_1 C_1}\\
	=A_1^2\pa{B^2-4 A_2 C_2}.
\end{multline*}
Therefore:
\begin{multline*}
\pa{A_1 B -A_2 B}^2 + A_2^2\pa{B^2-4 A_1 C_1} -A_1^2\pa{B^2-4 A_2 C_2}\\
	=\mp 2\pa{A_1 B -A_2 B}A_2\sqrt{B^2-4 A_1 C_1}.
\end{multline*}
Squaring both sides yields:
\begin{multline}\label{near_R}
\br{\pa{A_1 B -A_2 B}^2 + A_2^2\pa{B^2-4 A_1 C_1} -A_1^2\pa{B^2-4 A_2 C_2}}^2\\
	=4\pa{A_1 B -A_2 B}^2A_2^2\pa{B^2-4 A_1 C_1}.
\end{multline}
Subtract the right hand side of \eqref{near_R} to the left side, expand out by multiplication, and cancel the common factor $16 A_1^2 A_2^2\ne0$ to obtain:
\begin{equation}\label{nearer}
A_1^2 C_2^2 +A_2^2 C_1^2 -2A_1 A_2 C_1 C_2 +A_1 B^2 C_1 +A_2 B^2 C_2 - A_1 B^2 C_2 - A_2 B^2 C_1 = 0.
\end{equation}
At this point, use of a computer algebra system is recommended; we used MAGMA for our computations and Mathematica for independent verification.
In \eqref{nearer}, expand hyperbolics into exponentials and then expand out by multiplication.
Multiply through the resulting equation by $4096$ and then collect together terms whose exponentials have identical exponents as elements of $\Z\br{a,b,c,d}$.
This yields:
\begin{multline}\label{R_max_order}
\pa{w-1}^2 y^2 z^2 e^{\pa{4a+4b+2c+2d}} + \pa{x-1}^2 y^2 z^2 e^{\pa{2a+2b+4c+4d}} \\
 +2\pa{w-1}\pa{1-x} y^2 z^2 e^{\pa{3a+3b+3c+3d}} + \pa{\tn{sum of lower order terms}}=0.
\end{multline}
Equation \eqref{R_max_order} is our relation.
Its left hand side, denoted by $R$, is the sum of $1041$ terms each of the form $p e^{ka+lb+mc+nd}$ for some $p\in\Z\br{w,x,y,z}$ and some integers $k$, $l$, $m$, and $n$ in $\br{-4,4}$.
Here, a \emph{lower order} term is defined to be one whose exponential exponent $ka+lb+mc+nd$ is dominated by $4a+4b+2c+2d$, $2a+2b+4c+4d$, or $3a+3b+3c+3d$.
This means, by definition, that either:
\begin{enumerate}
\item $k\leq4$, $l\leq4$, $m\leq2$, $n\leq2$, and at least one of these inequalities is strict,
\item $k\leq2$, $l\leq2$, $m\leq4$, $n\leq4$, and at least one of these inequalities is strict, or
\item $k\leq3$, $l\leq3$, $m\leq3$, $n\leq3$, and at least one of these inequalities is strict.
\end{enumerate}

There are now two cases to consider.
\item Case 1. One of the sums $4a+4b+2c+2d$, $2a+2b+4c+4d$, or $3a+3b+3c+3d$ is greater than the other two.
Then, Lindemann's theorem implies that the corresponding polynomial coefficient in equation~\eqref{R_max_order} must vanish.
That is:
\[
\pa{w-1}^2 y^2 z^2=0, \quad \pa{x-1}^2 y^2 z^2=0, \quad \tn{or} \quad 2\pa{w-1}\pa{1-x} y^2 z^2=0.
\]
By equations~\eqref{wxyzE}, we have $w\ne1$, $x\ne1$, $y\ne0$, and $z\ne0$.
This contradiction completes the proof of Case 1.

\item Case 2. Two of the sums $4a+4b+2c+2d$, $2a+2b+4c+4d$, and $3a+3b+3c+3d$ are equal.
Then, $a+b=c+d$.
Recalling Figure~\ref{hyp_quad}, $a+b=c+d$ and two applications of the triangle inequality imply that $D$ cannot lie inside triangle $ABC$.
Similarly, $B$ cannot lie inside triangle $ADC$.
This implies that diagonal $AC$ lies inside quadrilateral $ABCD$.
Now, repeat the entire argument of the proof of Theorem~\ref{quad_thm} using diagonal $AC$ in place of diagonal $BD$.
We either obtain a contradiction as in Case 1 or we further obtain $a+d=b+c$.
Taken together, $a+b=c+d$ and $a+d=b+c$ imply that $a=c$ and $b=d$.
So, opposite sides of quadrilateral $ABCD$ are congruent.
By SSS, triangles $ABD$ and $CDB$ are congruent. In particular, $\delta_1=\beta_2$.
So, $\beta_1+\delta_1=\beta\in\Q\pi$ and triangle $ABD$ contradicts Lemma~\ref{alg_legs}.
This completes the proof of Claim 2 and of Theorem~\ref{quad_thm}.
\end{proof}

\section{Concluding Remarks}

The results above have analogues for generalized hyperbolic triangles with one ideal vertex and for spherical triangles.

\begin{lemma}
Let $ABC$ be a generalized hyperbolic triangle with one ideal vertex $C$ as in Figure~\ref{ideal_vertex}.
If $\alpha,\beta\in\Q\pi$, then the finite side length $c$ is transcendental.
\end{lemma}
\begin{figure}[htbp!]
    \centerline{\includegraphics[scale=1.0]{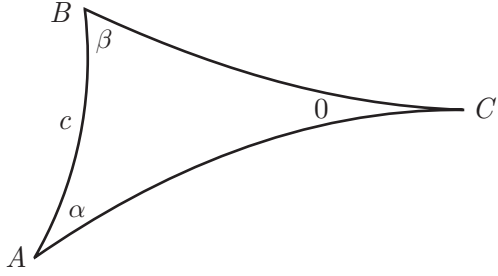}}
    \caption{Generalized hyperbolic triangle $ABC$ with one ideal vertex $C$.}
\label{ideal_vertex}
\end{figure}
\begin{proof}
By Ratcliffe~\cite[p.~88]{ratcliffe_2006}, we have:
\[
\cosh c =\frac{1+\cos\alpha \cos\beta}{\sin\alpha \sin\beta}.
\]
Thus:
\begin{equation}\label{ideal_rel}
\pa{\sin\alpha \sin\beta}e^c +\pa{\sin\alpha \sin\beta}e^{-c} = 2\pa{1+\cos\alpha \cos\beta}e^0.
\end{equation}
As $c>0$, Lindemann's theorem applied to~\eqref{ideal_rel} implies $c\notin\A$.
\end{proof}

Let $S^2\subset\R^3$ denote the unit sphere of constant Gaussian curvature $K=+1$.

\begin{lemma}
Let $ABC\subset S^2$ be a rational angled triangle labelled as in Figure~\ref{sph_triangle}.
Then, every side length of $ABC$ is transcendental. 
\end{lemma}
\begin{figure}[htbp!]
    \centerline{\includegraphics[scale=1.0]{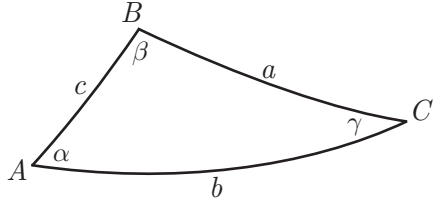}}
    \caption{Spherical triangle $ABC$.}
\label{sph_triangle}
\end{figure}

\begin{proof}
The second spherical law of cosines (see Ratcliffe~\cite[p.~49]{ratcliffe_2006}) gives:
\begin{equation}\label{spherical_rel}
\pa{\sin\alpha \sin\beta}e^{ic} + \pa{\sin\alpha \sin\beta}e^{-ic} = 2\pa{\cos\alpha \cos\beta + \cos\gamma}e^0
\end{equation}
where $c>0$ and $\sin\alpha \sin\beta\in\A-\cpa{0}$.
Whether or not the algebraic number $\cos\alpha \cos\beta + \cos\gamma$ vanishes, Lindemann's theorem applied to~\eqref{spherical_rel} implies $c\notin\A$.
The proofs for $a$ and $b$ are similar.
\end{proof}

On the other hand, the results in this note do not hold for arbitrary constant curvature.
In a plane of constant Gaussian curvature:
\[
K=-\pa{\cosh^{-1}\pa{1+\sqrt{2}}}^2 \approx -2.3365\ldots,
\]
a regular triangle with angles of radian measure $\pi/4$ has unit side lengths.
Similarly, in a sphere of constant Gaussian curvature:
\[
K=\frac{\pi^2}{4}\approx 2.4674 \ldots,
\]
a regular right angled triangle has unit side lengths.

\end{document}